\documentclass[11pt]{article}
\usepackage{amssymb,amsmath,amsfonts}
\usepackage{stmaryrd}
\usepackage{color}
\usepackage[latin1]{inputenc}
\usepackage{multirow}
\usepackage{subfigure}
\usepackage{enumerate}
\usepackage{epsfig}
\usepackage{epic}
\usepackage{pstricks}
\usepackage{pst-node}
\usepackage{pstricks-add}

\def\BBox{\kern  -0.2cm\hbox{\vrule width 0.2cm height 0.2cm}}

\newtheorem{theorem}{Theorem}[section]
\newtheorem{lemma}[theorem]{Lemma}

\newtheorem{definition}[theorem]{Definition}

\newtheorem{proposition}[theorem]{Proposition}
\newtheorem{remark}[theorem]{Remark}

\title{Families of small regular graphs of girth 7}

\author{ M. Abreu$^{1}$, G. Araujo-Pardo$^{2}$, C. Balbuena$^{3}$,
D. Labbate$^1$, J. Salas$^{3}$
\thanks{ Research   supported by the Ministerio de Educaci\'on y Ciencia,
Spain, the European Regional Development Fund (ERDF) under project
MTM2008-06620-C03-02;   and under the Catalonian Government
project 1298 SGR2009. CONACyT-M\'exico under project 57371 and
PAPIIT-M\'exico under project 104609-3. \newline \footnotesize{\em
Email addresses:} marien.abreu@unibas.it (M. Abreu),~
garaujo@matem.unam.mx (G. Araujo), ~ m.camino.balbuena@upc.edu (C.
Balbuena), \, \, ~ domenico.labbate@unibas.it (D. Labbate)}
 \\[2ex]
$^1${\footnotesize Dipartimento di Matematica, Informatica ed Economia, Universit\`{a} degli Studi della
Basilicata,}\\
{\footnotesize Viale dell'Ateneo Lucano, I-85100 Potenza, Italy.} \\$^2$
{\footnotesize Instituto de Matem\'{a}ticas, Universidad Nacional Aut\'{o}noma de M\'{e}xico,} \\
{\footnotesize M\'{e}xico D. F., M\'exico }\\
$^3${\footnotesize Departament de Matem\`atica Aplicada III, Universitat
Polit\`ecnica de Catalunya, }\\
{\footnotesize Campus Nord, Edifici C2, C/ Jordi Girona 1 i 3 E-08034 Barcelona,
Spain.} \\
}

\date{}

\begin{document}
\maketitle
\begin{abstract}
The first known families of cages arised from the incidence graphs of generalized polygons of order $q$, $q$ a prime power.
In particular,  $(q+1,6)$--cages have been obtained from the projective planes of order $q$. Morever, infinite families of small regular graphs of girth $5$ have been constructed performing algebraic operations on $\mathbb{F}_q$. 

In this paper, we introduce some combinatorial operations to construct new infinite families of small regular graphs of girth $7$ from the $(q+1,8)$--cages 
arising from the generalized quadrangles of order $q$, $q$ a prime power.
\end{abstract}

{\bf Keywords:}
Cages, girth, generalized quadrangles, latin squares.


\section{Introduction}

All graphs considered are finite, undirected and simple (without loops or
multiple edges). For definitions and notations not explicitly stated the reader may refer to
\cite{BM}, \cite{GR00} and \cite{LW94}.



 Let $G$ be a graph with vertex set
$V=V(G)$ and edge set
  $E=E(G)$.  The \emph{girth}  of a graph $G$ is the number $g=g(G)$ of edges in a
smallest cycle. For every $v\in V$, $N_G(v)$ denotes the \emph{neighbourhood} of $v$,
that is, the set of all vertices adjacent to $v$. The \emph{degree} of a vertex $v\in V$ is the
cardinality of    $N_G(v)$.   A graph is called
\emph{regular} if all the vertices have the same degree. A  \emph{$(k,g)$-graph} is a  $k$-regular graph with girth $g$.  Erd\H os and Sachs
 \cite{ES63}     proved the existence of  $(k,g)$-graphs
 for all values of $k$ and $g$ provided that $k \ge 2$. Thus most work carried
 out has  focused on constructing a smallest one
 \cite{AABL12-2,AFLN06,ABH10,BI73,B08,B09,B66,BMS95,E96,FH64,GH08,LU95,LUW97,M99,OW81,PBMOG04}.
A  \emph{$(k,g)$-cage} is a  $k$-regular graph with girth $g$ having the smallest possible
number of vertices. Cages have been  studied
intensely since they were introduced by
Tutte \cite{T47} in 1947.
Counting the numbers of vertices in the
distance partition with respect to a vertex yields a lower bound
$n_0(k,g)$ with the
precise form of the bound depending on whether $g$ is even or odd:

  \begin{equation}\label{lower} n_0(k,g) = \left\{ \begin{array}{ll} 1 + k + k(k-1) + \cdots
+ k(k-1)^{(g-3)/2} &\mbox{ if $g$ is odd};\\
2(1 +(k-1) +
\cdots + (k-1)^{g/2-1})&\mbox{ if $g$ is
even}.\end{array}\right.\end{equation}
 Biggs   \cite{B96}
 calls the \emph{excess} of a  $(k,g)$-graph $G$
the difference $|V(G)|-n_0(k,g)$. The    construction
of graphs with small excess is a difficult task.
Biggs is the author of a report on distinct methods for constructing  cubic cages
  \cite{B98}.
More details about constructions of cages can be found in the survey by Wong
\cite{W82} or in the  book  by Holton and Sheehan \cite{HS93} or in  the more
recent dynamic cage survey by Exoo and Jajcay \cite{EJ08}.

A  $(k,g)$-cage with $n_0(k,g)$ vertices and even girth  exist only when $g\in \{4,6,8,12\}$ \cite{FH64}. If $g=4$ they are the complete bipartite graph $K_{k,k}$, and for $g=6,8,12$ these graphs are the incidence graphs of generalized $g/2$-gons of order $k-1$.  This is the main reason for $(k,g)$-cages with
$n_0(k,g)$ vertices and even girth $g$  are called \emph{generalized polygon graphs}   \cite{B96}. In particular a $3$-gon of order $k-1$ is also known as a \emph{projective plane} of order $k-1$. The $4$-gons of order $k-1$ are called \emph{generalized quadrangles} of order $k-1$,
and, the $6$-gons of order $k-1$, \emph{generalized hexagons} of order $k-1$.  All these objets   are   known
to exist for all prime power values of $k-1$ \cite{B97,GR00,LW94}, and no example is known when $k-1$ is not a prime power.

In this article we focus on the case $g =8$. Let $q$ be a prime power. Our main objective is to give an explicit construction of small $(q+1,7)$--graphs for
$k=q+1 $. Next we present the contributions of this paper and in the following sections we do the corresponding proofs.


\section{Preliminaries}

It is well known \cite{PT, vanM} that $Q(4,q)$ and $W(3,q)$ are the only two classical generalized quadrangles with parameters $s=t=q$.
The generalized quadrangle $W(3,q)$ is the dual generalized of $Q(4,q)$, and they are selfdual for $q$ even.

In 1966 Benson \cite{B66} constructed $(q+1,8)$--cages from the generalized quadrangle
$Q(4,q)$. He defined the point/line incidence graph $\Gamma_q$ of $Q(4,q)$ which is a $(q+1)$--regular graph of girth $8$ with $n_0(q+1,8)$ vertices. Hence, $\Gamma_q$ is a  $(q+1,8)$--cage. Note that, $\Gamma_q$ is isomorphic to the point/line incidence graph of $W(3,q)$.

For any generalized quadrangle $Q$ of order $(s,t)$ and every point $x$ of
$Q$, let $x^\bot$ denote the set of all points collinear with $x$. Note that in the incidence graph $x^\bot=N_2(x)$, with an abuse of notation supposing that $x\in \Gamma_q$ corresponds to the point $x\in Q$

If $X$ is a nonempty set of vertices of $Q$, then we define $X^\bot :=\bigcap_{x\in X}x^\bot$.
The \emph{span} of the pair $(x,y)$ is $sp(x, y)=\{x, y\}^{\bot\bot}=\{u\in P: u\in z^\bot \forall z\in x^\bot \cap y^\bot \}$, where $P$ denotes the set of points in $Q$. If $x$ and $y$ are not
collinear, then $\{x,y\}^{\bot\bot}$ is also called the \emph{hyperbolic line} through $x$ and $y$.
If the \emph{hyperbolic line} through two noncollinear points $x$ and $y$ contains precisely $t+1$ points, then the pair $(x, y)$ is called \emph{regular}.
A point $x$ is called \emph{regular} if the pair $(x, y)$ is regular for every point $y$ not
collinear with $x$.
It is important to recall that the concept of regular also exists for a graph to avoid confusion. Hence we will emphasize when regular refers to a point or a graph.

\begin{remark}\label{regular}\cite{PT2}
Every point in $W(q)$ is regular.
\end{remark}


There are several equivalent coordinatizations of these generalized quadrangles (cf. \cite{P1}, \cite{U1}, \cite{U2}, see also \cite{vanM}) each giving a labeling for the graph $\Gamma_q$. Now we present a further labeling of $\Gamma_q$, equivalent to previous ones (cf. \cite{AABL12}), which will be central for our constructions since it allows us to keep track of the properties (such as regularity and girth) of the small regular graphs of girth $7$  obtained from $\Gamma_q$.

\begin{definition}\label{gammaq}
Let $\mathbb{F}_q$ be a finite  field with $q\ge 2$ a prime power. Let $\Gamma_q=\Gamma_q[V_0,V_{1}]$ be a bipartite graph with vertex sets
$V_r=\{(a,b,c)_r, (q,q,a)_r: a\in \mathbb{F}_q \cup  \{q\} ,b, c \in \mathbb{F}_q \}$, $r=0,1$, and
edge set defined as follows:
$$ \begin{array}{l}\mbox{For all } a\in \mathbb{F}_q\cup \{q\}  \mbox{ and for all } b,c\in \mathbb{F}_q:\\[2ex]
N_{\Gamma_q}((a,b,c)_{1} )= \left\{\begin{array}{ll}
    \{(x,~ax+b,~a^2x+2ab+c)_{0}: x \in \mathbb{F}_q \}\cup \{(q,a,c)_{0} \}  &\mbox{ if }   a\in \mathbb{F}_q;
\\[2ex]
\{( c ,b,x )_{0}: x \in \mathbb{F}_q \}\cup \{(q,q,c)_{0} \}  &\mbox{
if }  a= q.
\end{array}\right.
\\
\mbox{}\\
N_{\Gamma_q}((q,q,a)_{1})= \{(q,a,x)_{0}: x \in \mathbb{F}_q \}\cup
\{(q,q,q)_{0} \}.
\end{array}
$$
\end{definition}

Note that, in the labeling introduced in Definition \ref{gammaq}, the second $q$ in $\mathbb{F}_q \cup  \{q\}$, usually denoted by $\infty$, is meant to be just a symbol and no operations will be performed with it.

To finish, we define a \emph{Latin square} as an $n\times n$ array filled with $n$ different symbols, each occurring exactly once in each row and exactly once in each column.

In the following two sections we present our results only for $(q+1,8)$-cages, but all preliminary results are valid for all $(k,8)$-cages with any number $k$ given that they have the required combinatorial properties.
\section{Constructions of small $(q+1;7)$--graphs, for an even prime power $q$ }
In this section we will consider a $(q+1,8)$-cage $\Gamma_q$ with $q+1\geq 5$ an odd integer, since the only known $(q+1,8)$-cages are obtained as the incidence graph of a Generalized Quadrangles, we let  $q\geq 4$ a power of two.

Let $x\in V(\Gamma_q)$ and let $N(x)=\{x_0,....,x_q\}$, label $N(x_i)=\{x_{i0},x_{i1},...,x_{iq}=x\}$, for all $i\in \{0,...q\}$, in the following way.
Take $x_{0j}$ and $x_{1j}$ arbitrarily for $j=0,\ldots, q-1$ and let $N_2(x_{0j})\cap N_2(x_{1j})-x=W_j$, note that $|W_j|=q$.
Let $x_{ij}=(\displaystyle \bigcap_{w\in W_j} N_2(w))\cap N(x_i)$, these vertices exist and are uniquely labeled since the generalized quadrangle $W(q)$ is regular.

Let $H=x\cup N(x)\cup \{x_{q-1},x_q\}\cup \displaystyle \bigcup^{q-2}_0 N(x_i) \subset V(\Gamma_q)$.

We will delete the set $H$ of vertices of $\Gamma_q$ and add matchings $M_Z$ between the remaining neighbors of such vertices in order to obtain a small regular graph of girth 7.
In order to define the sets $M_Z$, we denote $X_i=N(x_i)\setminus  \{x\}$  and $X_{ij}=N(x_{ij})\setminus \{x_i\}$, for  $i\in \{0,...,q\}$ and $j\in \{0,...,q-1\}$.

Let $\mathcal{Z}$ be the family of all $X_{q-1}X_q,X_{ij}$ for $i\in \{0,...,q-2\}$ and $j\in \{0,...,q-1\}$.
For each $Z\in\mathcal{Z}$, $M_Z$ will denote a perfect matching of $V(Z)$, which will eventually be added to $\Gamma_q$.

\begin{definition}\label{g1}
Let $\Gamma_q$ be a  $(q+1,8)$-cage, with odd degree $q+1\ge 5$.

 Let $\Gamma_q1$ be the graph with:
$V(\Gamma_q1):=V(\Gamma_q-H) \mbox{ and }
\displaystyle E(\Gamma_q1):= E(\Gamma_q-H)\cup \bigcup_{Z\in\mathcal{Z}} M_Z.$

\end{definition}

Observe that the graph $\Gamma_q1$ has order $|V(\Gamma_q)|-(q^2+2)$ and all its vertices have degree $q+1$.

Next proposition states a condition for the graph $\Gamma_q1$ to have girth 7, for this it is useful to state the following remark.

\begin{remark}\label{length}
Let $u,v\in V(\Gamma_q)$ a graph of girth 8, such that there is a $uv$-path $P$ of length $t<8$. Then every $uv$-path $P'$ such that $E(P)\cap E(P')=\emptyset$ has length $|E(P')|\ge 8-t$.
\end{remark}

\begin{proposition}\label{girthg1}
Let $\Gamma_q$ be a  $(q+1,8)$-cage, with odd degree $q+1\ge 5$ and $\Gamma_q1$ as in Definition \ref{g1}. Then $\Gamma_q1$ has girth 7 if given $u_1v_1\in M_{X_{ij}}$ and $u_2,v_2\in X_{kl}$ such that $d(u_1,u_2)=2$ and $d(v_1,v_2)=2$, it holds $u_2v_2\not\in M_{X_{kl}}$, for $i\not= k\in \{0,...,q-2\}$ and $j,l\in \{0,...,q-1\}$.
\end{proposition}

\begin{proof}

Let us consider the distances (in $\Gamma_q-H$) between the elements in the sets $Z\in \mathcal{Z}.$
There are five possible cases:

\begin{enumerate}[(1)] \setlength\itemindent{0.6cm}
\item Two vertices in the same set $u,v\in Z$ have a common neighbor $w$ in $\Gamma_q$, therefore $d_{\Gamma_q-H}(u,v)\ge 6$.
\item If $u\in X_{q-1}$ and $v\in X_{q}$, then $d_{\Gamma_q-H}(u,v)\ge 4$, since $x_{q-1},x_q$ have $x$ as a common neighbor in $\Gamma_q$.
\item If $u\in X_i$ for $i\in \{q-1,q\}$ and $v\in X_{kj}$ for $k\in \{0,...,q-2\}$ and $j\in \{0,...,q-1\}$ then $d_{\Gamma_q}(u,x_i)=1$, $d_{\Gamma_q}(v,x_k)=2$, and $x_i,x_k$ have a common neighbor $x\in V(\Gamma_q)$, hence there is a $uv$-path of length 5 in $\Gamma_q$, concluding from Remark \ref{length} that $d_{\Gamma_q}(u,v)\ge 3$.
\item If $u\in X_{ij}$ and $v\in X_{ik}$ for $i\in \{0,...,q-2\}$ and $j,k\in \{0,...,q-1\}$, then $ux_{ij}x_{i}x_{ik}v$ is a path of length 4 and from Remark \ref{length} $d_{\Gamma_q-H}(u,v)\ge 4$.
\item If $u\in X_{ij}$ and $v\in X_{lk}$ for $i\not= l$, $i,l\in \{0,...,q-2\}$ and $j,k\in \{0,...,q-1\}$, then it is possible that there exist $w\in \Gamma_q-H$ such that $u,v\in N(w)$, that is $d_{\Gamma_q-H}(u,v)\ge 2$.

\end{enumerate}

\noindent Let us consider $C$ a shortest cycle in $\Gamma_q1$. If $E(C)\subset E(\Gamma_q-H)$ then $|C|\ge 8$.
Suppose $C$ contains edges in $\displaystyle M=\bigcup_{Z\in\mathcal{Z}} M_Z$.
If $C$ contains exactly one such edge, then by (1) $|C|\ge 7$.
If $C$ contains exactly two edges $e_1,e_2\in M$, the following cases arise.

\begin{enumerate}[.] \setlength\itemindent{0.6cm}
\item If both $e_1,e_2$ lie in the same $M_Z$ then by (1) $|C|\ge 14 >7$.
\item If $e_1\in M_{X_{q-1}}$ and $e_2\in M_{X_q}$ then by (2) $|C|\ge 10 >7$.
\item If $e_1\in M_{X_i}$ and $e_2\in M_{X_{kj}}$ then by (3) $|C|\ge 8 >7$.
\item If $e_1\in M_{X_{ij}}$ and $e_2\in M_{X_{ik}}$ then by (4) $|C|\ge 10 >7$.
\item If $e_1\in M_{X_{ij}}$ and $e_2\in M_{X_{lk}}$, for $i\not=l$, by hypothesis $|C|\ge 7$.
\end{enumerate}

\noindent If $C$ contains at least three edges of $M$, since $d(u,v)\ge 2$ for all $u,v\in \{X_{q-1}, X_q,X_{ij}\}$ with $i\in \{0,...,q-2\}$ and $j\in \{0,...,q-1\}$, $|C|\ge 9>7$.
\newline Hence $\Gamma_q1$ has girth 7 and we have finished the proof.
\end{proof}

The following lemma gives sufficient conditions to define the matchings $M_{X_{ij}}$ for the sets $X_{ij}$, for $i\in \{0,...,q-2\}$ and $j\in \{0,...,q-1\}$, in order to fulfill the condition from Proposition \ref{girthg1}.

\begin{lemma}\label{match}
There exist $q^2-q$ matchings $M_{X_{ij}}$, for each $i\in \{0,...,q-2\}$ and $j\in \{0,...,q-1\}$ with the following property:

Given $u_1v_1\in M_{X_{ij}}$ and $u_2,v_2\in X_{kj}$ such that $d(u_1,u_2)=2$ and $d(v_1,v_2)=2$ then $u_2v_2\not\in M_{X_{kj}}$.
\end{lemma}
\begin{proof}
By definition $\displaystyle \bigcap^{q-2}_{i=0} N(X_{ij})=W_j$. Let $W_j=\{w_{j1},\ldots,w_{jq}\}$. Note that every vertex $w_{jh}$ is adjacent to exactly one vertex in $N(X_{ij})$ that we will denote as $x_{ijh}$, for each $i\in \{0,...,q-2\}$ and $j\in \{0,...,q-1\}$.

Observe that $x_{ijh}$ is well defined, because if $x_{ijh}$ had two neighbors $w_h,w_{h'}\in \bigcap^{q-2}_{i=0} N(X_{ij})$, $\Gamma_q$ would contain the cycle $x_{ijh}w_{jh'}x_{i'jh'}x_{i'j}x_{i'jh}w_{jh}$ of length 6.

Therefore, take the complete graph $K_{q}$ label its vertices as $h=1,\ldots,q$. We know that it has a 1-factorization with $q-1$ factors $F_1,\ldots,F_{q-1}$.
For each $i=0,\ldots,q-2$, let $x_{ijh}x_{ijh'}\in M_{X_{ij}}$ if and only if $hh'\in F_{i}$.

To prove that the matchings $M_{X_{ij}}$ defined in this way fulfill the desired property suppose that $x_{ijh}x_{ijh'}\in M_{X_{ij}}$ and $x_{i'jh}x_{i'jh'}\in M_{X_{i'j}}$ for $i'\not=i$, then $F_i$ and $F_{i'}$ would have the edge $hh'$ in common contradicting that they are a factorization.

Therefore, there exist $q^2-q$ matchings $M_{X_{ij}}$ with the desired property.
\end{proof}

To finish, notice that for $u_1v_1\in M_{X_{ij}}$ and $u_2,v_2\in X_{i'j'}$ with $j\not=j'$ and possibly $i=i'$, the distances $d(u_1,u_2)$ and $d(v_1,v_2)$ are at least 4.
Then, counting the number of vertices of $\Gamma_q1$ and using the Proposition \ref{girthg1} we have the following theorem.

\begin{theorem}\label{main}
Let $q\ge 4$ be a power of two. Then there is a $(q+1)$-regular graph of girth 7 and order $2q^3+q^2+2q$.
\end{theorem}

\section{Constructions of small $(q+1;7)$--graphs for and odd prime power $q$.}

In this section we will consider cages of even degree, that $\Gamma_q$ is a $(q+1,8)$-cage with $q$ an odd prime power. We proceed as before, but as will be evident from the proofs, the result is not as good as in the previous section.

We will delete a set $H$ of vertices of $\Gamma_q$ and add matchings $M_Z$ between the remaining neighbors of such vertices in order to obtain a small regular graph of girth 7.
The sets $H$ and $M_Z$ are defined as follows.

Let $V=\{x,y\}\cup \{s_0,\ldots,s_{q}\}$ be the vertices of $K_{2,q+1}$.

Let $\widehat{K_{2,q+1}}$ be the graph obtained subdividing each edge of $K_{2,q+1}$.

Let $\Gamma_q$ be a graph containing a copy of $\widehat{K_{2,q+1}}$ as a subgraph and label its vertices as $H'=\{x,y,s_0,\ldots,s_{q}\}\cup N(x)\cup N(y)$
where $N(x)=\{x_0,\ldots,x_q\}$ and $N(y)=\{y_0,\ldots,y_q\}$.
Note that $N(x_i)\cap N(y_i)=s_i$ for $i=0,\ldots, q$.
Define:
$$
  \begin{array}{lll}
    H & = & \{x,y,s_3,s_4\cdots,s_{q}\}\cup N(x)\cup N(y)\subset V(\Gamma_q);   \\
    X_i & = & N(x_i)\cap V(\Gamma_q-H), \quad i=0,\ldots, q; \\
    Y_i & = & N(y_i)\cap V(\Gamma_q-H), \quad i=0,\ldots, q; \\
    S_i & = & N(s_i)\cap V(\Gamma_q-H), \quad i=3,\ldots, q. \\
  \end{array}
$$

Notice that the vertices of $\Gamma_q-H$ have degrees $q - 1, q$ and $q + 1$.
The vertices $s_0,s_1,s_2$ of degree $q - 1$, those in $X_i\cup Y_i\cup S_i$ of degree $q$ and all the remaining vertices of $\Gamma_q-H$ have degree $q+1$.
Therefore, in order to complete the degrees to such vertices its necessary to add edges to $\Gamma_q-H$, we define such edges next.

Let $\mathcal{Z}$ be the family of all $X_i,Y_i,S_i$.
For each $Z\in\mathcal{Z}$, $M_Z$ will denote a perfect matching of $V(Z)$, which will eventually be added to $\Gamma_q$.

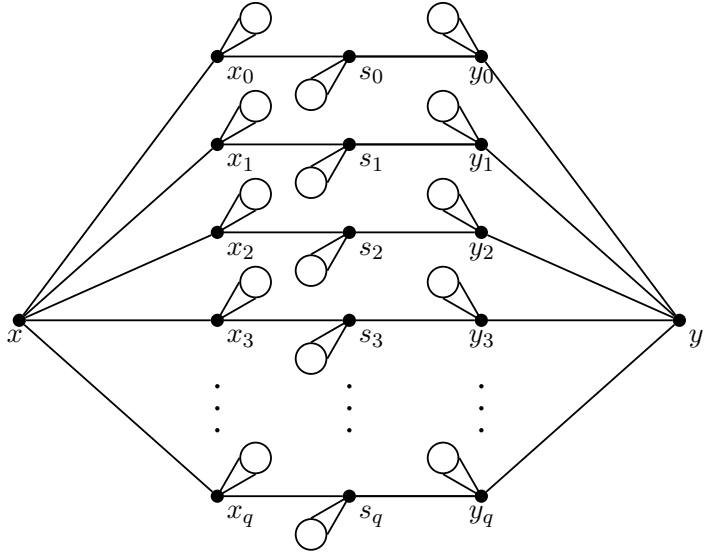
\resizebox{10cm}{!}{
\begin{minipage}[h]{10cm}
\begin{center}

\psset{unit=0.6cm, linewidth=0.025cm}
  \begin{pspicture}(0,-6)(24,7)

\cnode*(1.5   ,   0)   {0.15}   {  x }\put( 1   ,   -0.5   ){ $x$}

\cnode*(6   ,   6)   {0.15}   {  x1 }\put( 6   ,   5.5   ){ $x_0$}
\cnode*(6   ,   4)   {0.15}   {  x2 }\put( 6   ,   3.5   ){ $x_1$}
\cnode*(6   ,   2)   {0.15}   {  x3 }\put( 6   ,   1.5   ){ $x_2$}
\cnode*(6   ,   0)   {0.15}   {  x4 }\put( 6   ,   -0.5   ){ $x_3$}

\cnode*(6   ,   -1.5)   {0.05}   {  x5 }
\cnode*(6   ,   -2)   {0.05}   {  x6 }
\cnode*(6   ,   -2.5)   {0.05}   {  x7 }

\cnode*(6   ,  -4)   {0.15}   {  xk }\put( 6   ,   -4.5   ){ $x_q$}

\cnode*(9   ,   6)   {0.15}   {  s1 }\put( 9   ,   5.5   ){ $s_0$}
\cnode*(9   ,   4)   {0.15}   {  s2 }\put( 9   ,   3.5   ){ $s_1$}
\cnode*(9   ,   2)   {0.15}   {  s3 }\put( 9   ,   1.5   ){ $s_2$}
\cnode*(9   ,   0)   {0.15}   {  s4 }\put( 9   ,   -0.5   ){ $s_3$}

\cnode*(9   ,   -1.5)   {0.05}   {  s5 }
\cnode*(9   ,   -2)   {0.05}   {  s6 }
\cnode*(9  ,   -2.5)   {0.05}   {  s7 }

\cnode*(9   ,   -4)   {0.15}   {  sk }\put( 9  ,   -4.5   ){ $s_q$}

\cnode*(12   ,   6)   {0.15}   {  y1 }\put( 11.5   ,   5.5   ){ $y_0$}
\cnode*(12   ,   4)   {0.15}   {  y2 }\put( 11.5   ,   3.5   ){ $y_1$}
\cnode*(12   ,   2)   {0.15}   {  y3 }\put( 11.5   ,   1.5   ){ $y_2$}
\cnode*(12   ,   0)   {0.15}   {  y4 }\put( 11.5   ,   -0.5   ){ $y_3$}

\cnode*(12   ,   -1.5)   {0.05}   {  y5 }
\cnode*(12   ,   -2)   {0.05}   {  y6 }
\cnode*(12  ,   -2.5)   {0.05}   {  y7 }

\cnode*(12   ,   -4)   {0.15}   {  yk }\put( 11.5  ,   -4.5   ){ $y_q$}

\cnode*(16.5  ,   0)   {0.15}   {  y }\put(16.5   ,   -0.5   ){ $y$}

\ncline[]{-}{x}{x1} \ncline[]{-}{x}{x2} \ncline[]{-}{x}{x3} \ncline[]{-}{x}{x4}  \ncline[]{-}{x}{xk}

\ncline[]{-}{x1}{s1} \ncline[]{-}{x2}{s2} \ncline[]{-}{x3}{s3} \ncline[]{-}{x4}{s4} \ncline[]{-}{xk}{sk}

\ncline[]{-}{y1}{s1} \ncline[]{-}{y2}{s2} \ncline[]{-}{y3}{s3} \ncline[]{-}{y4}{s4} \ncline[]{-}{yk}{sk}

\ncline[]{-}{s1}{y1} \ncline[]{-}{s2}{y2} \ncline[]{-}{sk}{yk}

\ncline[]{-}{y}{y1} \ncline[]{-}{y}{y2} \ncline[]{-}{y}{y3} \ncline[]{-}{y}{y4}  \ncline[]{-}{y}{yk}

\psline[]{-}(6,6)(6.87,6.5) \psline[]{-}(6,6)(6.5,6.87)
\pscircle(6.87,6.87){0.37}
\psline[]{-}(9,6)(8.13,5.5) \psline[]{-}(9,6)(8.5,5.13)
\pscircle(8.13,5.13){0.37}
\psline[]{-}(12,6)(11.13,6.5) \psline[]{-}(12,6)(11.5,6.87)
\pscircle(11.13,6.87){0.37}

\psline[]{-}(6,4)(6.87,4.5) \psline[]{-}(6,4)(6.5,4.87)
\pscircle(6.87,4.87){0.37}
\psline[]{-}(9,4)(8.13,3.5) \psline[]{-}(9,4)(8.5,3.13)
\pscircle(8.13,3.13){0.37}
\psline[]{-}(12,4)(11.13,4.5) \psline[]{-}(12,4)(11.5,4.87)
\pscircle(11.13,4.87){0.37}

\psline[]{-}(6,2)(6.87,2.5) \psline[]{-}(6,2)(6.5,2.87)
\pscircle(6.87,2.87){0.37}
\psline[]{-}(9,2)(8.13,1.5) \psline[]{-}(9,2)(8.5,1.13)
\pscircle(8.13,1.13){0.37}
\psline[]{-}(12,2)(11.13,2.5) \psline[]{-}(12,2)(11.5,2.87)
\pscircle(11.13,2.87){0.37}

\psline[]{-}(6,0)(6.87,0.5) \psline[]{-}(6,0)(6.5,0.87)
\pscircle(6.87,0.87){0.37}
\psline[]{-}(9,0)(8.13,-0.5) \psline[]{-}(9,0)(8.5,-0.87)
\pscircle(8.13,-0.87){0.37}
\psline[]{-}(12,0)(11.13,0.5) \psline[]{-}(12,0)(11.5,0.87)
\pscircle(11.13,0.87){0.37}

\psline[]{-}(6,-4)(6.87,-3.5) \psline[]{-}(6,-4)(6.5,-3.13)
\pscircle(6.87,-3.13){0.37}
\psline[]{-}(9,-4)(8.13,-4.5) \psline[]{-}(9,-4)(8.5,-4.87)
\pscircle(8.13,-4.87){0.37}
\psline[]{-}(12,-4)(11.13,-3.5) \psline[]{-}(12,-4)(11.5,-3.13)
\pscircle(11.13,-3.13){0.37}
     \end{pspicture}
\end{center}
\end{minipage}
}

\begin{definition}\label{g1g2}
Let $\Gamma_q$ be a  $(q+1,8)$-cage, with even degree $q+1\ge 6$.
\begin{itemize}
\item Let $\Gamma_q1$ be the graph with:
$V(\Gamma_q1):=V(\Gamma_q-H) \mbox{ and }
\displaystyle E(\Gamma_q1):= E(\Gamma_q-H)\cup \bigcup_{Z\in\mathcal{Z}} M_Z.$

\item Define $\Gamma_q2$ as
$V(\Gamma_q2):=V(\Gamma_q1)$ and

$E(\Gamma_q2):= (E(\Gamma_q1)\setminus \{u_0v_0,u_1v_1,u_2v_2\})\cup \{s_0u_0,s_0v_0,s_1u_1,s_1v_1,s_2u_2,s_2v_2\},$

 where $s_i\in H'-H$,
the deleted edges $u_iv_i$ belong to $M_{X_i}$ in $\Gamma_q1$ and they are replaced by the paths of length two $u_is_iv_i$, $i=0,1,2$.
\end{itemize}

\end{definition}

By an immediate counting argument we know that the graph $\Gamma_q1$ has order $|V(\Gamma_q)|-3(q+1)+1$, and observe that all vertices in $\Gamma_q1$ have degree $q+1$ except for $s_0s_1,s_2$ which remain of degree $q-1$. Hence, by the definition of $E(\Gamma_q2)$, all vertices in $\Gamma_q2$ are left with degree $q+1$.

\begin{proposition}\label{struct}
Let $\Gamma_q$ be a  $(q+1,8)$-cage, with even degree $q\ge 5$ and $\Gamma_q1$, $\Gamma_q2$ be as in Definition \ref{g1g2}.

\begin{itemize}
\item[(i)] $\Gamma_q1$ has girth 7 if the matchings $M_{S_i},M_{X_i}$ and $M_{Y_i}$ have the following properties:
    \begin{itemize}
    \item[(a)]Given $u_1v_1\in M_{S_i}$ and $u_2,v_2\in S_j$ such that $d(u_1,u_2)=2$ and $d(v_1,v_2)=2$, it holds that $u_2v_2\not\in M_{S_j}$.
    \item[(b)]Given $u_1v_1\in M_{X_i}$ and $u_2,v_2\in Y_j$ such that $d(u_1,u_2)=2$ and $d(v_1,v_2)=2$, it holds that $u_2v_2\not\in M_{Y_j}$.
    \end{itemize}

\item[(ii)]
    If conditions (a) and (b) hold then the graph $\Gamma_q2$ also has girth 7.

\end{itemize}

\end{proposition}
\begin{proof}
To prove $(i)$ let us consider the distances (in $\Gamma_q-H$) between the elements in the sets $Z\in \mathcal{Z}.$
There are six possible cases:

\begin{enumerate}[(1)] \setlength\itemindent{0.6cm}
\item Two vertices in the same set $u,v\in Z$ have a common neighbor $w$ in $\Gamma_q$, therefore $d_{\Gamma_q-H}(u,v)\ge 6$.
\item If $u\in X_i$ and $v\in X_j$ then $d_{\Gamma_q-H}(u,v)\ge 4$, given that $x_i,x_j$ have $x$ as a common neighbor in $\Gamma_q$.
\item If $u\in Y_i$ and $v\in Y_j$ then $d_{\Gamma_q-H}(u,v)\ge 4$, as before.
\item If $u\in S_i$ and $v\in S_j$ then it is possible that there exist $w\in \Gamma_q-H$ such that $u,v\in N(w)$, that is, $d_{\Gamma_q-H}(u,v)\ge 2$.
\item If $u\in S_i$ and $v\in X_j\cup Y_j$ then $d_{\Gamma_q-H}(u,v)\ge 3$, since $s_i\in N(x_i)\cap N(y_i)$.
\item If $u\in X_i$ and $v\in Y_j$ then $d_{\Gamma_q-H}(u,v)\ge 2$.
\end{enumerate}

\noindent Let us consider $C$ a shortest cycle in $\Gamma_q1$. If $E(C)\subset E(\Gamma_q-H)$ then $|C|\ge 8$.
Suppose $C$ contains edges in $\displaystyle M=\bigcup_{Z\in\mathcal{Z}} M_Z$.
If $C$ contains exactly one such edge, then by (1) $|C|\ge 7$.
If $C$ contains exactly two edges $e_1,e_2\in M$, the following cases arise:

\begin{enumerate}[.] \setlength\itemindent{0.6cm}
\item If both $e_1,e_2$ lie in the same $M_Z$, then by (1) $|C|\ge 14 >7$.
\item If $e_1\in M_{X_i}$ and $e_2\in M_{X_j}$ for $i\not=j$, by (2) $|C|\ge 10 >7$.
\item If $e_1\in M_{Y_i}$ and $e_2\in M_{Y_j}$ for $i\not=j$, by (3) $|C|\ge 10 >7$.
\item If $e_1\in M_{S_i}$ and $e_2\in M_{X_j}\cup M_{Y_j}$, by (5) $|C|\ge 8 >7$.
\item If $e_1\in M_{S_i}$ and $e_2\in M_{S_j}$ for $i\not=j$, by the first hypothesis in item (i)(b) $|C|\ge 7$.
\item If $e_1\in M_{X_i}$ and $e_2\in M_{Y_j}$, by the second hypothesis in item (i)(b) $|C|\ge 7$.
\end{enumerate}

\noindent If $C$ contains at least three edges of $M$, since $d(u,v)\ge 2$ for all $u,v\in \{X_i\cup Y_i\}^{k}_{i=1}\cup \{S_i\}^{k}_{i=4}$, $|C|\ge 9>7$.
\newline Hence $\Gamma_q1$ has girth 7, concluding the proof of $(i)$.

To prove $(ii)$, let $C$ be a shortest cycle in $\Gamma_q2$. If $E(C)\subset E(\Gamma_q-H)\cup M$ then $|C|\ge 7$.

\begin{enumerate}[.] \setlength\itemindent{0.6cm}
\item If $C$ contains exactly one edge $s_iu_i$ or $s_iv_i$ then $|C|\ge 7$ since $d_{\Gamma_q}(s_i,u_i)=d_{\Gamma_q}(s_i,v_i)=2$ which implies $d_{\Gamma_q1}(s_i,u_i)\ge 6$ and $d_{\Gamma_q1}(s_i,v_i)\ge 6$.

\item If $C$ contains a path $u_is_iv_i$ then $(C\setminus u_is_iv_i)\cup u_iv_i$ is a cycle in $\Gamma_q1$ with one less vertex than $C$, therefore $|C|\ge 8$.

\item If $C$ contains two edges $s_iu_i$, $s_ju_j$, for $i\not=j$. Their distances $d_{\Gamma_q1}(s_i,u_j)\ge 4$, $d_{\Gamma_q1}(s_i,s_j)\ge 4$, and $d_{\Gamma_q1}(u_i,u_j)\ge 4$, therefore in any case $C$ has length greater than 7 concluding the proof.
\end{enumerate}
\end{proof}


The following lemma gives sufficient conditions to define the matchings $M_{S_i}$ for the sets $S_i$, in order that they fulfill condition $(a)$ from Proposition \ref{struct} (i). Notice that in the incidence graph of a generalized quadrangle $\{x,y\}^{\bot\bot}=\bigcap _{s\in N_2(x)\cap N_2(y)}N_2(s)$, thus Remark \ref{regular} implies that $\displaystyle|\bigcap^{q}_{i=0} N(S_i)|=q-1$, recalling that $\{s_i\}_{i=0}^q=N_2(x)\cap N_2(y)$.
Since $\displaystyle|\bigcap^{q}_{i=0} N(S_i)|$ is contained in $\displaystyle|\bigcap^{q}_{i=3} N(S_i)|$, and $\displaystyle|\bigcap^{q}_{i=3} N(S_i)|\le |S_i|=q-1$ then the condition for the following lemma holds.

\begin{lemma}\label{eses}
If $\displaystyle|\bigcap^{q}_{i=3} N(S_i)|=q-1$ then there exist matchings $M_{S_i}$, for $i=3,\ldots,{q}$, such that:
    \begin{itemize}
    \item[-]Given $u_1v_1\in M_{S_i}$ and $u_2,v_2\in S_j$ such that $d(u_1,u_2)=2$ and $d(v_1,v_2)=2$, it holds that $u_2v_2\not\in M_{S_j}$.
    \end{itemize}
\end{lemma}
\begin{proof}
Let us suppose that $\bigcap^{q}_{i=3} N(S_i)=\{w_1,\ldots,w_{q-1}\}$, and since $S_i$ has $q-1$ vertices, every vertex $w_j$ is adjacent to exactly one vertex in $s_{ij}\in S_i$.

Observe that $s_{ij}$ is well defined, because if $s_{ij}$ had two neighbors $w_j,w_{j'}\in \bigcap^{q+1}_{i=1} N(S_i)$, $\Gamma_q$ would contain the cycle $(s_{ij}w_js_{kj}s_ks_{kj'}w_{j'})$ of length 6.

Therefore, take the complete graph $K_{q-1}$, label its vertices as $j=1,\ldots,q-1$. We know that it has a 1-factorization with $q-2$ factors $F_1,\ldots,F_{q-2}$.
For each $i=3,\ldots,q+1$, let $s_{ij}s_{il}\in M_{S_i}$ if and only if $jl\in F_{i-3}$.

To prove that the matchings $M_{S_i}$ defined in this way fulfill the desired property suppose that $s_{ij}s_{il}\in M_{S_i}$ and $s_{i'j}s_{i'l}\in M_{S_i'}$ for $i'\not=i$. Then $F_i$ and $F_{i'}$ would have the edge $jl$ in common contradicting that they were a factorization.
\end{proof}

So far, the steps of our construction have been independent from the coordinatization of the chosen $(q+1,8)$-cage, however, in order to define $M_{X_i}$ and $M_{Y_i}$ satisfying condition (b) of Lemma \ref{struct}, we need to fix all the elements chosen so far.

We will distinguish two cases, when $q$ is a prime or when $q$ is a prime power.

Choose $x=(q,q,q)_1$, $y=(0,0,0)_1$.

When $q$ is a prime then $x_{i}=(q,q,i)_0$, $y_{i}=(i,0,0)_0$ for $i=0,\ldots, q$.

Therefore, $N(x_i)=\{(q,t,i)_1:t=0,\ldots, q-1\}\cup x$ and $N(x_q)=\{(q,q,t)_1:t=0,\ldots q-1\}\cup x$; $N(y_i)=\{(t,-{it},{i+t^2})_1:t=0,\ldots q-2\}\cup (q,0,i)_1$ and $N(y_q)=\{(0,t,0)_1:t=0,\ldots q-1\}\cup (q,q,0)_1$.

Thus, the corresponding vertices $s_i$ are: $s_i=(q,0,i)_1$ for $i=0,\ldots q-1$ and $s_q=(q,q,0)_1$;
$N(s_i)=\{(i,0,t)_0:t=1,\ldots, q-1,i=0,\ldots,q\} \cup \{x_i,y_i\}$. Hence,
$S_i=\{(i,0,t)_0:t=1,\ldots, q-1,i=0,\ldots,q\}$.

Then $N(S_i)=\{(a,b,c)_1: b=-ia, c=t+a^2{i}, i=0,\ldots,q-1 \}$, and $N(S_q)=\{(q,0,t)_1:t=0,\ldots,q-1\}$.

Solving the equations we obtain $N(S_i)\cap N(S_j)=\{(0,0,t)_1: t=0,\ldots,q-1\}$, moreover $N(i,0,t)_0\cap N(j,0,t)_0=(0,0,t)_1$, for each $j\not=i$ and $t=0,\ldots,q-1$, or equivalently, $N(0,0,t)_1=\{(x,0,t)_0:t=0,\ldots,q-1,x=0,\ldots,q\}$. Hence the sets $S_i$ satisfy the hypothesis of Lemma \ref{eses}, yielding that there exist the matchings $M_{S_i}$ with the desired property.

Notice that the sets $X_i$ and $Y_i$ are naturally defined as the sets $X_i=\{(q,t,i)_1:t=1,\ldots, q-1,i=0,\ldots, q-1\}$, $X_0=\{(q,t,0)_1:t=1,\ldots, q-1\}$ and $X_q=\{(q,q,t)_1:t=1,\ldots, q-1\}$. The sets $Y_i=\{(t,-{it},{it^2})_1:t=1,\ldots, q-1,i=0,\ldots, q-1\}$, and $Y_q=\{(0,t,0)_1:t=1,\ldots, q-1\}$.

In this way we have defined all the sets in Lemma \ref{struct}, and from Lemma \ref{eses} we know that the matchings $M_{S_i}$ have the property that:
\begin{itemize}
    \item[-]If $u_1v_1\in M_{S_i}$ and $u_2,v_2\in S_j$ are such that $d(u_1,u_2)=2$ and $d(v_1,v_2)=2$ then $u_2v_2\not\in M_{S_j}$.
\end{itemize}

It remains to define the matchings $M_{X_i}$ and $M_{Y_i}$ and prove they have property $(b)$ from Proposition \ref{struct} $(i)$.

For this we must analyze the intersection of the second neighborhood of an $X_j$ with an $Y_i$, $N_2(X_j)\cap Y_i$. For each $w\in Y_i$, we know there is exactly one $z\in X_q$ such that $w\in N_2(z)$.

This allows us to define the following sets of latin squares:
For each $j$, let the coordinate $i\ell $ of the $j$-th latin square to have the symbol $s_{i\ell j}$ if there is a $w_{i\ell j}=(a,b,c)_1$ such that
$$w_{i\ell j}\in N((i,0,0)_0)\cap N_2((q,\ell ,j)_1)\cap N_2((q,q,s_{i\ell j})_1),$$
 where $(i,0,0)_0=y_i$, $(q,\ell ,j)_1\in X_j$ and $(q,q,s_{i\ell j})_1\in X_q$.

Since $N((i,0,0)_0)=\{(t,-{it},{i+t^2})_1:t=0,\ldots q-2\}\cup (q,0,{i})_1$, then $a=t$, $b=-it$, and $c=i+t^2$.

Observe that $w_{i\ell j}\in N_2((q,\ell ,j)_1)$ is equivalent to $(j,\ell ,t)_0\in N((a,b,c)_1)$, since $N((q,\ell ,j)_1)=\{(j,\ell ,t)_0:t=0,\ldots q-1\}\cup\{(q,q,j)_0\}$. Hence, $aj+b=\ell $.

And $w_{i\ell j}\in N_2((q,q,s_{i\ell j})_1)$ implies $a=s_{i\ell j}$.

Therefore we obtain the following equation for $s_{i\ell j}$.

$$s_{i\ell j}(j-i)=\ell $$

Notice that this equation is undefined for $j=i$, otherwise it would mean that $y_i$ has a neighbor at distance 3 from $x_j$ and this would imply the existence of a cycle of length 6 in $\Gamma_q$.

Also from the equation we deduce that $-s_{i\ell j}=s_{i-\ell j}$, and $s_{i+1\ell j+1}=s_{i\ell j}$.
This means that the $i+1$-th row of the $j+1$-th latin square is equal to the $i$-th row of the $j$-th latin square, hence all the set of latin squares have the same rows.
This also implies that if we put an edge between two vertices on $Y_i$, $(s_{i\ell j},-{is_{i\ell j}},{is_{i\ell j}^2})_1$ and $(-s_{i\ell j},{is_{i\ell j}},is_{i\ell j}^2)_1$, it will have at distance two in $X_j$ only the vertices $(q,\ell ,i)_1$ and $(q,-\ell ,i)_1$.

Therefore, the matchings $M_{X_i}= \{(q,\ell ,i)_1(q,-(\ell+2) ,i)_1:i=0,\ldots q-1,\ell =1,...,q-3\}\cup \{(q,-2,i)_1(q,-1,i)_1:i=0,\ldots q-1\}$, $M_{X_q}=\{(q,q,\ell )_1(q,q,-(\ell+2) )_1:\ell =1,\ldots, q-3\}\cup \{(q,q,-2 )_1(q,q,-1)_1\}$, and $M_{Y_i}=\{(t,-{it},{it^2})_1(-t,{it},it^2)_1:i=0,\ldots, q-1,t=1,\ldots, q-1\}$, have the property $(b)$ from Proposition \ref{struct} $(i)$.

When $q$ is a prime power, let $\alpha$  a primitive root of unity in $GF(q)$.
Then,
$x_{i}=(q,q,\alpha^{i-1})_0$, $y_{i}=(\alpha^{i-1},0,0)_0$ for $i=1,\ldots q-1$, $x_0=(q,q,0)_0$, and $y_{0}=(0,0,0)_0$. Moreover, $x_q=(q,q,q)_0$ and $y_{q}=(q,0,0)_0$.

Therefore, $N(x_i)=\{(q,\alpha^t,\alpha^{i-1})_1:t=0,\ldots q-2\}\cup (q,0,\alpha^{i-1})_1\cup x$ and $N(x_0)=\{(q,\alpha^t,0)_1:t=0,\ldots q-2\}\cup (q,0,0)_1\cup x$; $N(y_i)=\{(\alpha^t,-\alpha^{i-1+t},\alpha^{i-1+2t})_1:t=0,\ldots q-2\}\cup (q,0,\alpha^{i-1})_1$ and $N(y_0)=\{(\alpha^t,0,0)_1:t=0,\ldots q-2\}\cup (q,0,0)_1$; $N(x_q)=\{(q,q,\alpha^t)_1:s=0,\ldots q-2\}\cup (q,q,0)_1\cup x$; and $N(y_q)=\{(0,\alpha^t,0)_1:t=0,\ldots q-2\}\cup (q,q,0)_1\cup y$.

Thus, the corresponding vertices $s_i$ are: $s_i=(q,0,\alpha^{i-1})_1$, for $i=1,\ldots q-1$, $s_0=(q,0,0)_1$ and $s_q=(q,q,0)_1$;
$N(s_i)=\{(\alpha^{i-1},0,\alpha^t)_0:t=0,\ldots q-2,i=1,\ldots,q-1\} \cup \{x_i,y_i\}$, and $N(s_0)=\{(0,0,\alpha^t)_0:t=0,\ldots q-2\} \cup \{x_0,y_0\}$. Hence $S_i=\{(\alpha^{i-1},0,\alpha^t)_0:t=0,\ldots q-2,i=0,\ldots,q\}$ and $S_0=\{(0,0,\alpha^t)_0:t=0,\ldots q-2\}$.

Then $N(S_i)=\{(a,b,c)_1: b=-\alpha^{i-1}a, c=\alpha^t+a^2\alpha^{i-1}, i=1,\ldots,q-1 \}$, $N(S_0)=\{(a,b,c)_1: b=0, c=\alpha^t\}$ and $N(S_q)=\{(q,0,\alpha^t)_1:t=0,\ldots,q-2\}\cup (q,0,0)_1$.

Solving the equations we obtain $N(S_i)\cap N(S_j)=\{(0,0,\alpha^t)_1: t=0,\ldots,q-2\}$. Moreover, $N(\alpha^{i-1},0,\alpha^t)_0\cap N(\alpha^{j-1},0,\alpha^t)_0=(0,0,\alpha^t)_1$, for each $j\not=i$ and $t=0,\ldots,q-2$, or equivalently, $N(0,0,\alpha^t)_1=\{(\alpha^x,0,\alpha^t)_0:x=0,\ldots,q-2\}\cup (0,0,\alpha^t)_0\cup (q,0,\alpha^t)_0$, for each $t=0,\ldots,q-2$. Hence the sets $S_i$ satisfy the hypothesis of Lemma \ref{eses} yielding that there exist the matchings $M_{S_i}$ with the desired property.

Notice that the sets $X_i$ and $Y_i$ are naturally defined as the sets $X_i=\{(q,\alpha^{t},\alpha^{i-1})_1:t=0,\ldots, q-2,i=1,\ldots, q-1\}$, $X_0=\{(q,\alpha^{t},0)_1:t=0,\ldots, q-2\}$ and $X_q=\{(q,q,\alpha^{t})_1:t=0,\ldots q-2\}$. The sets

$Y_i=\{(\alpha^t,-\alpha^{i-1+t},\alpha^{i-1+2t})_1:t=0,\ldots, q-2\}$, $Y_0=\{(\alpha^t,0,0)_1:t=0,\ldots, q-2\}$ and $Y_q=\{(0,\alpha^t,0)_1:t=0,\ldots, q-2\}$.

In order to define the matchings $M_{X_i}$ and $M_{Y_i}$ and prove that they have the property $(b)$ from Proposition \ref{struct} $(i)$,
we proceed as before, by defining the sets of latin squares:

For each $j$, let the coordinate $i\ell $ of the $j$-th latin square to have the symbol $s_{i\ell j}\in\{0,\ldots, q-2\}$ if there is a $w_{i\ell j}=(a,b,c)_1$ such that
$$w_{i\ell j}\in N((\alpha^{i-1},0,0)_0)\cap N_2((q,\alpha^\ell,\alpha^{j-1})_1)\cap N_2((q,q,\alpha^{s_{i\ell j}})_1)  \mbox{ for $i,j\ge 1$,} $$
 where $(\alpha^{i-1},0,0)_0=y_i$, $(q,\alpha^\ell,\alpha^{j-1})_1\in X_j$ and $(q,q,\alpha^{s_{i\ell j}})_1\in X_q$.

Since $N((\alpha^{i-1},0,0)_0)=\{(\alpha^t,-\alpha^{i-1+t},\alpha^{i-1+2t})_1:t=0,\ldots q-2\}\cup (q,0,{i})_1$, then $a=\alpha^t$, $b=-\alpha^{i-1+t}$, and $c=\alpha^{i-1+2t}$.

Also $w_{i\ell j}\in N_2((q,\alpha^\ell,\alpha^{j-1})_1)$ is equivalent to $(\alpha^{j-1},\alpha^\ell ,\alpha^t)_0\in N((a,b,c)_1)$, since $N((q,\alpha^\ell,\alpha^{j-1})_1)=\{(\alpha^{j-1},\alpha^\ell ,\alpha^t)_0:t=0,\ldots q-2\}$. Hence
$a\alpha^{j-1}+b=\alpha^\ell $.

And $w_{i\ell j}\in N_2((q,q,\alpha^{s_{i\ell j}})_1)$ implies $a=\alpha^{s_{i\ell j}}$.

Therefore we obtain the following equation for $s_{i\ell j}$.

$$\alpha^{s_{i\ell j}}(\alpha^{j-1}-\alpha^{i-1})=\alpha^{\ell} $$

Notice that this equation is undefined for $j=i$, otherwise it would mean that $y_i$ has a neighbor at distance 3 from $x_j$ and this would imply the existence of a cycle of length 6 in $\Gamma_q$.

For $i=0$, we obtain the equation $\alpha^{s_{0\ell j}}(\alpha^{j-1})=\alpha^{\ell}$, and for $j=0$, we obtain $\alpha^{s_{i\ell 0}}(-\alpha^{i-1})=\alpha^{\ell}$.
From the equation we obtain that $s_{i\ell+1 j}=s_{i\ell j}+1$, and each latin square is the sum table of the cyclic group $\mathbb{Z}_{q-1}$ with the rows permuted.

Multiplying by $\alpha$ the equation $\alpha^{s_{i\ell-1 j}}(\alpha^{j-1}-\alpha^{i-1})=\alpha^{\ell-1}$, we obtain that $s_{i+1\ell j+1}=s_{i\ell-1 j}$.
This implies that the row $i+1$ of the $j+1$-th latin square is equal to the row $i$ of the $j$-th latin square subtracting 1 to each symbol (i.e.,  $s_{i+1\ell j+1}+1=s_{i\ell j}$).
That is, all the set of latin squares have the same rows but in a different order.

This also implies that if we put an edge between two vertices on $Y_i$, $(\alpha^{s_{i\ell j}},-\alpha^{i-1+s_{i\ell j}},\alpha^{i-1+2s_{i\ell j}})_1$ and $(\alpha^{s_{i\ell j}+1},-\alpha^{i-1+(s_{i\ell j}+1)},\alpha^{i-1+2(s_{i\ell j}+1)})_1$, it will have at distance two in $X_j$ only the vertices, $(q,\alpha^\ell ,i)_1$ and $(q,\alpha^{\ell +1},i)_1$ and the other way around.

Therefore, the matchings $M_{X_i}= \{(q,\alpha^{2\ell} ,i)_1(q,\alpha^{2\ell+1},i)_1:i=0,\ldots q-1,\ell =1,...,(q-1)/2\}$, $M_{X_q}=\{(q,q,\alpha^{2\ell} )_1(q,q,\alpha^{2\ell+1} )_1:\ell =1,\ldots, (q-1)/2\}$, and $M_{Y_i}=\{(\alpha^{2t},-\alpha^{i-1+2t},\alpha^{i-1+4t})_1(\alpha^{2t+3},-\alpha^{i-1+(2t+3)},\alpha^{i-1+2(2t+3)})_1:i=0,\ldots q-1,t=1,...,(q-1)/2\}$ have the property $(b)$ from Proposition \ref{struct} $(i)$, proving the theorem for $q$ prime power.
\begin{theorem}\label{main}
Let $q\ge 5$ be a prime power. Then there is a $q+1$-regular graph of girth 7 and order $2q^3+2q^2-q+1$.  
\end{theorem}

\begin{proof}
Finally, by applying Lemma \ref{struct}(ii), we obtain a $q+1$-regular graph of girth 7 with $2(q^3+q^2+q+1)-(q-3+2(q+2))=2q^3+2q^2-q+1$ vertices.
\end{proof}

\end{document}